\pdfoutput=1
\documentclass[noinfoline]{imsart}
\usepackage{amsmath}
\usepackage{bbm}
\usepackage{amsfonts}
\usepackage{amssymb}
\usepackage{amsthm}
\usepackage{subcaption}
\usepackage{dsfont}

\usepackage{enumitem}

\usepackage{fullpage}

\usepackage{algorithm}
\usepackage{algpseudocode}

\newtheorem{thm}{Theorem}

\newtheorem{lemma}{Lemma}
\newtheorem{cor}{Corollary}
\newtheorem{defin}{Definition}
\newtheorem{assumption}{Assumption}

\newcommand{\Ccal}{\mathcal{C}}
\newcommand{\Gcal}{\mathcal{G}}

\newcommand{\Ncal}{\mathcal{N}}

\def\R{\mathbb{R}}
\def\E{\mathbb{E}} 
\DeclareMathOperator*{\Ind}{\mathds{1}}

\def\ba{\boldsymbol{a}}

\def\be{\boldsymbol{e}}

\def\bt{\boldsymbol{t}}
\def\bu{\boldsymbol{u}}
\def\bv{\boldsymbol{v}}

\def\bx{\boldsymbol{x}}

\def\bz{\boldsymbol{z}}

\def\bA{\boldsymbol{A}}
\def\bB{\boldsymbol{B}}
\def\bC{\boldsymbol{C}}
\def\bD{\boldsymbol{D}}
\def\bE{\boldsymbol{E}}

\def\bI{\boldsymbol{I}}

\def\bM{\boldsymbol{M}}

\def\bQ{\boldsymbol{Q}}
\def\bR{\boldsymbol{R}}

\def\bU{\boldsymbol{U}}

\def\bX{\boldsymbol{X}}
\def\bY{\boldsymbol{Y}}
\def\bZ{\boldsymbol{Z}}

\def\ones{\boldsymbol{1}}

\def\bbeta{\boldsymbol{\beta}}
\def\bmu{\boldsymbol{\mu}}
\def\btheta{\boldsymbol{\theta}}
\def\beps{\boldsymbol{\varepsilon}}
\def\bzeta{\boldsymbol{\zeta}}

\def\bDelta{\boldsymbol{\Delta}}
\def\bLambda{\boldsymbol{\Lambda}}
\def\bPi{\boldsymbol{\Pi}}
\def\bPsi{\boldsymbol{\Psi}}
\def\bSigma{\boldsymbol{\Sigma}}
\def\bTheta{\boldsymbol{\Theta}}

\def\diag{\textup{diag}}
\def\R{\mathbb{R}}
\def\E{\mathbb{E}}

\DeclareMathOperator*{\cov}{Cov}

\DeclareMathOperator*{\var}{Var}

\DeclareMathOperator{\Tr}{Tr}
\DeclareMathOperator*{\card}{card}

\DeclareMathOperator*{\minimize}{minimize}

\DeclareMathOperator*{\argmin}{argmin}
\DeclareMathOperator*{\Pen}{Pen}
\DeclareMathOperator*{\pr}{pr}

\newcommand{\bpm}{\begin{pmatrix}}
\newcommand{\epm}{\end{pmatrix}}

\usepackage{pdfsync}

\RequirePackage[OT1]{fontenc}
\RequirePackage[numbers,sort]{natbib}
\RequirePackage[colorlinks,citecolor=blue,urlcolor=blue]{hyperref}
\RequirePackage{hypernat}
\setattribute{journal}{name}{}

\begin{document}
\begin{frontmatter}
\title{Prediction and estimation consistency of sparse multi-class penalized optimal scoring}
\runtitle{Consistency of penalized optimal scoring}

\begin{aug}
\author{\fnms{Irina Gaynanova
}\ead[label=e1]{irinag@stat.tamu.edu}}


\address{Department of Statistics\\
Texas A\&M University\\	
MS 3143\\	
College Station, TX 77843\\
\printead{e1}}

\runauthor{Gaynanova}


\affiliation{Department of Statistics at Texas A\&M University}

\end{aug}

\begin{abstract}~
Sparse linear discriminant analysis via penalized optimal scoring is a successful tool for classification in high-dimensional settings. While the variable selection consistency of sparse optimal scoring has been established, the corresponding prediction and estimation consistency results have been lacking. We bridge this gap by providing probabilistic bounds on out-of-sample prediction error and estimation error of multi-class penalized optimal scoring allowing for diverging number of classes.
\end{abstract}


\end{frontmatter}

\section{Introduction}\label{sec:intro}

Sparse linear discriminant analysis has been proven to be a successful tool for classification in high-dimensional settings \citep{Cai:2011dm, Witten:2011kc, Mai:2012bf, Clemmensen:2011kr}. While multiple formulations have been proposed,
exploiting the connection between the linear discriminant analysis and optimal scoring problem~\citep{Hastie:1994cx, Hastie:1995gg} leads to a particularly attractive regularization due to the least squares loss function.

Let $(\bx_i, \bz_i)$, $i=1,\dots, n$, be independent pairs, where $\bx_i\in \R^p$ is a vector of features, and $\bz_i\in \{0,1\}^K$ is a vector indicating class membership, $z_{ik}=1$ if $i$th sample belongs to class $k\in \{1,\dots, K\}$ and $z_{ik}=0$ otherwise.
Let $\bX\in \R^{n \times p}$ be a column-centered data matrix, $\bZ\in \R^{n\times K}$ be the corresponding class indicator matrix and $n_k$ be the number of samples in class $k$. Let ${\bf{1}}\in \{1\}^{K}$ be a vector of ones. 
The unpenalized optimal scoring problem \citep{Hastie:1994cx} is formulated as 
\begin{align*}
\minimize_{\bTheta, \bB} \|\bZ\bTheta-\bX \bB\|^2_F&\\
\mbox{subject to} \quad n^{-1}\bTheta^{\top}\bZ^{\top}\bZ\bTheta&=\bI_{K-1},\quad \bTheta^{\top}\bZ^{\top}\bZ\bf{1}=0,
\end{align*}
where $\bB = [\bbeta_1\dots \bbeta_{K-1}]\in \R^{p \times (K-1)}$ is the matrix of feature coefficients, and $\bTheta \in \R^{K\times(K-1)}$ is the matrix of scores.
It is shown in \citep{Hastie:1994cx} that linear discriminant analysis can be carried out by solving unpenalized optimal scoring problem.

In the special case of two classes, $K=2$,
the solution for the vector of scores is $\widehat \btheta = (\sqrt{n_2/n_1}, -\sqrt{n_1/n_2})^{\top}$ up to a sign. Defining $\bY=\bZ\widehat \btheta$, the optimal scoring problem reduces to the linear regression problem. Given the success of lasso \citep{Tibshirani:1996wb} in high-dimensional linear regression, \citep{Mai:2012bf} consider the penalized optimal scoring problem
\begin{equation}\label{eq:opscore}
\widehat \bbeta = \argmin_{\bbeta}\{(2n)^{-1}\|\bY - \bX \bbeta\|^2_2 + \lambda \|\bbeta\|_1\}.
\end{equation}
Further generalizations to the copula models \citep{Han:2013ju}, tensor data \citep{Pan:2018ff} as well as the multi-class case \citep{Clemmensen:2011kr,Wu:2015vb,Gaynanova:2016wk, Merchante:2012vk} have been considered.

While the prediction and estimation consistency of lasso estimator in linear regression has been well-studied, see for example
\citep{Greenshtein:2004iz, Bunea:2007gr, Bickel:2009eb, Zhang:2009ck, Rigollet:2011fz,  Dalalyan:2017je} and references therein, the theoretical analysis of~\eqref{eq:opscore} and related extensions have been primarily focused on variable selection consistency \citep{Mai:2012bf,Kolar:2013vs, Gaynanova:2015km, Gaynanova:2016wk, Han:2013ju, Pan:2018ff}. The latter requires the use of irrepresentable condition \citep{Zhao:2006vy}, which significantly limits the amount of correlations allowed in $\bX$, and is more restrictive than conditions needed for the prediction consistency \citep{VanDeGeer:2009we}.

There are several reasons for the gap between theoretical understanding of sparse optimal scoring and lasso. First, the linear model for $\bY$ does not hold as the covariates in $\bX$ are generated conditionally on the class membership encoded by $\bY$. Secondly, since the covariates are random, it is of interest to investigate the expected out-of-sample prediction risk rather than in-sample prediction risk typically considered in linear regression literature \citep{Bickel:2009eb,Dalalyan:2017je, Hastie:2015wu}. Specifically, let $\bx\in \R^p$ be a new vector of covariates with the same distribution as $\bx_i$, and for $K=2$ let $\bbeta^*$ be the population matrix of coefficients, then the expected out-of-sample prediction risk is defined as
$$
R(\widehat \bbeta) = \E_{\bx}\{\|\bx^{\top}(\widehat \bbeta - \bbeta^*)\|_2^2\} = (\widehat \bbeta - \bbeta^*)^{\top}\E(\bx\bx^{\top})(\widehat \bbeta - \bbeta^*),
$$
whereas the in-sample prediction risk is defined as $\|\bX(\widehat \bbeta - \bbeta^*)\|_2^2$. Finally, defining the residual terms as $\beps:=\bY - \bX \bbeta^*$, the residuals in $\beps$ and the covariates in $\bX$ are not independent. These challenges prevent direct application of lasso results to~\eqref{eq:opscore}.

In this work we address these challenges and bridge the existing gap in theoretical understanding of sparse optimal scoring. Specifically, our work makes the following contributions:

- Compared to existing research specific to $K=2$ case \citep{Li:2017kb}, we consider a multi-class framework, and show that the matrix of optimal scores $\widehat \bTheta$ can be expressed in a closed form up to an orthogonal transformation (Lemma~\ref{l:tildetheta}). This allows us to formulate a coordinate-sparse multi-class optimal scoring problem as the penalized multiple response linear regression problem, thus enabling the subsequent theoretical analysis. We believe this result is of independent interest.

- We derive the concentration bound for the maximal row $\ell_2$ norm of $n^{-1}\bX^{\top}\bE$ (Theorem~\ref{t:Xtepsilon}), where $\bE:=\bY - \bX\bB^*$ is the matrix of residuals. The key difficulties in deriving this bound is the non-gaussianity of $\bX$ and $\bE$, and the lack of independence between $\bX$ and $\bE$. The corresponding proof is the key theoretical contribution of this work.

- We derive out-of-sample prediction and estimation bounds for sparse multi-class optimal scoring problem which  allow both the number of features $p$ and the number of classes $K$ to grow with the sample size $n$. The corresponding results for the estimator in~\eqref{eq:opscore} follow as a special case when $K=2$. We derive bounds of two types, that are typically called slow-rate bounds and fast-rate bounds in the literature, we refer to \citep{Dalalyan:2017je, Bien:2018jt} for the discussion. Slow-rate bounds make no assumptions on the correlation structure of $\bX$ or the sparsity of the population matrix of coefficients $\bB^*$, whereas fast-rate bounds lead to faster convergence rates, but rely on exact sparsity of $\bB^*$ and restricted eigenvalue condition \citep{Bickel:2009eb}. 



\subsection{Relations to Existing Literature} 

The variable selection consistency of estimator in~\eqref{eq:opscore} has been established in \citep{Mai:2012bf,Kolar:2013vs}, whereas the variable selection consistency for the estimator in the multi-class case has been established in \citep{Gaynanova:2015km, Gaynanova:2016wk}. While the estimation consistency can be established under the same conditions, the proofs rely on irrepresentability condition. To our knowledge the results on prediction and estimation consistency of~\eqref{eq:opscore} without irrepresentability condition are lacking, with the exception of a recent work by Li and Jia~\citep{Li:2017kb}.

In \citep{Li:2017kb}, Li and Jia establish $\ell_2$ estimation consistency of penalized optimal scoring when $K=2$. Our results and analysis differ in several ways. Most importantly, we consider the multi-class case, and allow the number of classes $K$ to grow with the sample size. This generalization is far from trivial, and requires establishing score invariance (Lemma~\ref{t:poptim}), derivation of the explicit form of the scores (Lemma~\ref{l:tildetheta}) as well as a new proof of the concentration bound for $n^{-1}\bX^{\top}\bE$ term (Theorem~\ref{t:Xtepsilon}). Theorem~\ref{t:Xtepsilon} applies to the two-class case as well, but our proof allows to explicitly characterize the dependence of constants on model parameters and is significantly reduced compared to the proof in \citep{Li:2017kb}. Secondly, in addition to $\ell_2$ consistency, we establish the bounds on expected out-of-sample prediction error, where expectation is taken with respect to a new vector of features $\bx\in \R^p$. Similar distinction is made in \citep{Chatterjee:2013vs}, where the difference between mean squared prediction error and estimated mean squared prediction error is discussed. The out-of-sample prediction bounds are not present in~\citep{Li:2017kb}, largely due to the latter focus on fast-rate bounds. In contrast, we derive both fast-rate and slow-rate bounds. The main advantage of the slow-rate bounds is that they do not require either sparsity assumption or the restricted eigenvalue condition, we refer to \citep{Dalalyan:2017je, Bien:2018jt} for the discussion of the two types of bounds. As part of the slow-rate bound derivation, we demonstrate that the norm of $\widehat \bB$ can always be bounded by a constant times the norm of $\bB^*$. While the proof is rather simple, we found that this fact was not explicitly stated in the literature, and therefore could be of independent interest.  Finally, the $\ell_2$ estimation consistency in \citep{Li:2017kb} is established explicitly under the restricted eigenvalue condition on $\bX$.  While the authors state that ``a few class of matrices have been proved to satisfy the restricted eigenvalue condition with high probability" and refer to \citep{Raskutti:2010vm} for corresponding results for gaussian designs, these resullrs are neither incorporated into the analysis nor is $\bX$ gaussian in optimal scoring. We show that the entries of $\bX$ are marginally sub-gaussian with explicit characterization of sub-gaussian constant (Lemma~\ref{l:Xsubgaussian}), and correspondingly rely on results of \citep{Zhou:2009wba,Rudelson:2013jw} to establish restricted eigenvalue condition with high probability. We also incorporate these bounds within the analysis.

\subsection{Notation}
 For two scalars $a, b \in \R$, we let $a \vee b = \max(a,b)$. For a vector $\bv\in \R^p$, we define $\ell_1$-norm as $\|\bv\|_1 = \sum_{i=1}^p|v_i|$, $\ell_2$-norm as $\|\bv\|_2 = (\sum_{i=1}^p v_i^2)^{1/2}$ and $\ell_{\infty}$ norm as $\|\bv\|_{\infty} = \max_i |v_i|$. We use $\ones\in \R^p$ to denote a vector of ones, ${\bf 0}\in \R^p$ to denote a vector of zeros, and $\be_j\in \R^p$ to denote a unit-norm vector with $j$th coordinate equal to one. For scalar $a\in \R$, we write $\{a\}_l$ to denote a row-vector of length $l$ with each element equal to $a$. For a matrix $\bA \in \R^{n \times p}$, we let $\|\bA\|_{\infty,2}=\max_{i}(\sum_{j=1}^pa_{ij}^2)^{1/2}$, $\|\bA\|_{1,2}=\sum_{i=1}^{n}(\sum_{j=1}^pa_{ij}^2)^{1/2}$, $\|\bA\|_1 = \sum_{i=1}^n\sum_{j=1}^p|a_{ij}|$, $\|\bA\|_2 = \sup_{\bx: \|\bx\|_2 = 1}\|\bA \bx\|_2$, $\|\bA\|_F =(\sum_{i=1}^n\sum_{j=1}^pa_{ij}^2)^{1/2}$ and $\|\bA\|_{\infty} = \max_{i,j}|a_{ij}|$. We use $\bI$ to denote the identity matrix. For a sequence of scalars $b_1,\dots, b_n,\dots$, we use $b_n = o(a_n)$ if $\lim_{n\to \infty}(b_n/a_n)=0$ and $b_n = O(a_n)$ if $\lim_{n \to \infty}(b_n/a_n) = C$ for some finite constant $C$. For a sequence of random variables $x_1, \dots, x_n, \dots$, we use $b_n = o_p(a_n)$ if for any $\varepsilon>0$ $P(|b_n|/a_n<\varepsilon)\to 0$ as $n \to \infty$, and $b_n = O_p(a_n)$ if for any $\varepsilon > 0$ there exists $M_{\varepsilon}$ such that $P(|b_n|/a_n > M_{\varepsilon})<\varepsilon$ for all $n$. For random variable $t$, we use $\|t\|_{\psi_2} = \sup_{p\geq 1}p^{-1/2}(\E|t|^p)^{1/p}$ for sub-gaussian norm of $t$, and $\|t\|_{\psi_1} = \sup_{p\geq 1}p^{-1}(\E|t|^p)^{1/p}$ for sub-exponential norm of $t$.
 
\subsection{Paper organization}

The rest of the manuscript is organized as follows. In Section~\ref{s:scoring} we consider penalized optimal scoring for the multi-class case, and demonstrate that coordinate-sparse multi-class optimal scoring problem can be formulated as a multiple response penalized linear regression problem. In Section~\ref{s:deterministic} we derive deterministic bounds for expected out-of sample prediction error and $\ell_2$ estimation error of sparse optimal scoring. In Section~\ref{s:prob} we derive concentration bound for the maximal row $\ell_2$ norm of $n^{-1}\bX^{\top}\bE$, which subsequently allows us to derive probabilistic slow-rate and fast-rate bounds. In Section~\ref{sec:discus} we conclude with discussion. All the proofs are deferred to Appendix.

\section{Multi-class penalized optimal scoring}\label{s:scoring}

We consider multi-class penalized optimal scoring problem 
\begin{equation}\label{eq:penopscore}
\begin{split}
\minimize_{\bTheta, \bB} &\{(2n)^{-1}\|\bZ\Theta-\bX \bB\|^2_F+\lambda \Pen(\bB)\}\\
\mbox{subject to}& \quad \bTheta^{\top}\bZ^{\top}\bZ\bTheta=n\bI_{K-1},\quad \bTheta^{\top}\bZ^{\top}\bZ\bf{1}=0,
\end{split}
\end{equation}
where $\Pen(\bB): \R^{p \times (K-1)} \to [0,\infty)$ is a penalty function. For example, \citep{Hastie:1994cx} uses $\Pen(\bB) = \Tr(\bB^{\top}\bA \bB)$ for some positive definite matrix $\bA$, \citep{Clemmensen:2011kr} use $\Pen(\bB) = \|\bB\|_1$, and for $K=2$ \citep{Mai:2012bf} use $\Pen(\bbeta)=\|\bbeta\|_1$. We first show that if the penalty function is invariant with respect to orthogonal transformation, then any matrix of scores within the constraint set will lead to global solution of~\eqref{eq:penopscore}.

\begin{lemma}\label{t:poptim} Let $\Pen(\bB) = \Pen(\bB\bR)$ hold for any $\bB\in \R^{p \times (K-1)}$ and any orthogonal matrix $\bR\in \R^{(K-1) \times (K-1)}$. Let $\widetilde \bTheta\in \R^{K\times (K-1)}$ be such that $\widetilde\bTheta^{\top}\bZ^{\top}\bZ\widetilde\bTheta=n\bI_{K-1}$, $\widetilde\bTheta^{\top}\bZ^{\top}\bZ\ones=\bf{0}$, and let 
\begin{equation}\label{eq:scoringKt}
\bB_{\widetilde\bTheta} = \argmin_{\bB} \left\{(2n)^{-1}\|\bZ\widetilde\bTheta - \bX \bB\|^2_F + \lambda \Pen(\bB)\right\}.
\end{equation}
Then the pair $(\widetilde\bTheta, \bB_{\widetilde\bTheta})$ attains global minimum of~\eqref{eq:penopscore}.
\end{lemma}

Since any matrix $\widetilde\bTheta$ that satisfies the constraints leads to the pair $(\widetilde\bTheta, \bB_{\widetilde \bTheta})$ that minimizes the objective function, we next show that such a matrix can be constructed explicitly based on the sample sizes $n_k$.

\begin{lemma}\label{l:tildetheta} Let $\widetilde \bTheta\in \R^{K \times (K-1)}$ have columns $\widetilde \bTheta_l\in \R^{K}$, $l=1,\dots,K-1$, defined as
$$
\widetilde \bTheta_l=\Big(\Big\{\sqrt{\frac{nn_{l+1}}{\sum_{i=1}^ln_i\sum_{i=1}^{l+1}n_i}}\Big\}_{l},\quad-\sqrt{\frac{n\sum_{i=1}^ln_i}{n_{l+1}\sum_{i=1}^{l+1}n_i}},\quad \{0\}_{K-1-l}\Big)^{\top}.
$$
Then $\widetilde\bTheta^{\top}\bZ^{\top}\bZ\widetilde\bTheta=n\bI_{K-1}$ and $\widetilde\bTheta^{\top}\bZ^{\top}\bZ\ones={\bf 0}$.
\end{lemma}

Lemmas~\ref{t:poptim} and~\ref{l:tildetheta} show that to solve a penalized optimal scoring problem with orthogonally invariant penalty function, it is sufficient to fix the scores according to Lemma~\ref{l:tildetheta}, and only consider problem~\eqref{eq:scoringKt}, which has the same form as a penalized multiple response linear regression problem. The condition on the penalty function is satisfied by many commonly-used penalties, for example by $\Pen(\bB) = \Tr(\bB^{\top}\bA \bB)$ for any $K$ and by $\Pen(\bbeta)=\|\bbeta\|_1$ for $K=2$. When $K>2$, the element-wise sparsity penalty $\Pen(\bB) = \|\bB\|_1$ does not satisfy the condition, however the coordinate-wise sparsity penalty  $\Pen(\bB) = \|\bB\|_{1,2}$ does. While the difference between element-wise and coordinate-wise sparsity may seem minor, we argue that coordinate-wise sparsity is preferable in the discriminant analysis context. Similar argument is made in \citep{2016arXiv160505918B} for the principal component analysis. When $\Pen(\bB) = \|\bB\|_1$, each column of $\bB_{\widetilde \Theta}$ is sparse, however the rows are not necessarily sparse which means both that the individual features are not completely eliminated from the classification rule, and that the sparsity is not preserved under orthogonal transformation. In contrast, when $\Pen(\bB) = \|\bB\|_{1,2}$, all columns of $ \bB_{\widetilde \Theta}$ share the same sparsity pattern leading to sparse rows, and consequently feature elimination. Therefore, we let $\bY = \bZ\tilde \bTheta$ and define the solution to sparse multi-class penalized optimal scoring as
\begin{equation}\label{eq:scoringK}
\widehat \bB = \argmin_{\bB}\{(2n)^{-1}\|\bY-\bX \bB\|_F^2 + \lambda \|\bB\|_{1,2}\}.
\end{equation}
When $K=2$,~\eqref{eq:scoringK} reduces to~\eqref{eq:opscore}. When $K>2$,~\eqref{eq:scoringK} can be rewritten in a form equivalent to sparse linear discriminant analysis proposal of \citep{Gaynanova:2016wk}, although the latter does not draw connections to optimal scoring. In the rest of the paper, we derive bounds on expected out-of-sample prediction risk and estimation error of estimator in~\eqref{eq:scoringK}.

\section{Deterministic bounds}\label{s:deterministic}

In this section we derive out-of-sample prediction and $\ell_2$ estimation bounds for~\eqref{eq:scoringK} that hold deterministically under certain conditions on $\bX$ and $\lambda$. We first review the explicit form of the matrix of population discriminant vectors $\bB^*$.

Let $\bx_i \in \Ccal_k$ denote that sample $i$ belongs to class $k\in\{1,\dots, K\}$, and let $\pi_k = \pr(\bx_i \in \Ccal_k)$, $\bmu_k = \E(\bx_i| \bx_i \in \Ccal_k)$, $\bSigma_W = \cov(\bx_i|\bx_i \in \Ccal_k)$. Let $\bSigma_T$ be the marginal covariance matrix such that $\cov(\bx_i) = \bSigma_T$. \citep{Gaynanova:2016wk} show that the population matrix of canonical vectors can be expressed as $\bSigma_T^{-1}\bDelta \bR$, where $\bDelta\in \R^{K\times(K-1)}$ is the matrix of orthogonal mean contrasts between $K$ classes with $r$th column defined as
\begin{equation}\label{eq:delta}
\bDelta_r = \frac{\sqrt{\pi_{r+1}}\sum_{k=1}^r\pi_k(\bmu_k - \bmu_{r+1})}{\sqrt{\sum_{k=1}^r\pi_k\sum_{k=1}^{r+1}\pi_{k}}},
\end{equation}
and $\bR$ is the $(K-1)\times (K-1)$ orthogonal matrix of eigenvectors of $\bDelta^{\top}\bSigma_T^{-1}\bDelta$. Moreover, since the classification rule is invariant to orthogonal transformations, any orthogonal matrix $\bR$ will lead to equivalent classification rule. The orthogonal invariance of $\bB^*$ mimics the orthogonal invariance explored in Lemma~\ref{t:poptim}, which is not by chance. Our choice of $\widetilde \bTheta$ in Lemma~\ref{l:tildetheta} is such that $\E(n^{-1}\bX^{\top}\bZ\widetilde\bTheta) = \bDelta + o(1)$ (see Lemma~\ref{l:D} in the Appendix), and we fix $\bB^* = \bSigma_T^{-1}\bDelta$ throughout the manuscript. A different choice of $\widetilde \bTheta$ leads to equivalent conclusions by applying corresponding orthogonal transformation to $\bB^*$. In a special case of two classes, $\bbeta^* = \bSigma_T^{-1}\bDelta_1 = \sqrt{\pi_1\pi_2} \bSigma_T^{-1}(\bmu_1-\bmu_2)$, which coincides with discriminant analysis direction considered in the literature \citep{Cai:2011dm,Mai:2012bf}.

Let $\bx\in \R^p$ be a new vector of covariates with the same distribution as $\bx_i$. Given $\bB^* = \bSigma_T^{-1}\bDelta$, we aim to derive bounds on expected out-of-sample prediction error defined as
$$
R(\widehat \bB) :=\E_{\bx}\|\bx^{\top}(\bB^* - \widehat \bB)\|_F^2 =\Tr\{(\bB^* - \widehat \bB)^{\top}\bSigma_T(\bB^* - \widehat \bB)\},
$$
and the estimation error $\|\widehat \bB-\bB^*\|_F^2$. Throughout, we define residual terms as $\bE = [\beps_1\dots \beps_{K-1}]:= \bY-\bX \bB^* $, which allows direct comparison with lasso bounds. Since~\eqref{eq:scoringK} is of the same form as penalized multiple-response linear regression problem with group-lasso penalty, the proofs of deterministic bounds for~\eqref{eq:scoringK} follow the proofs of deterministic bounds for lasso with extra triangle inequality to handle out-of-sample rather than in-sample prediction error. Because these bounds are deterministic, we follow the terminology in \citep{Lederer:2016wa} by differentiating \textit{penalty} and \textit{sparsity} deterministic bounds. In Section~\ref{s:prob}, we use these bound to derive probabilistic \textit{slow-rate} and \textit{fast-rate} bounds correspondingly.

\subsection{Penalty bounds}\label{sec:penalty}

We start by providing a deterministic bound for in-sample prediction error $\|\bX(\widehat \bB - \bB^*)\|_F^2$. This bound makes no assumption on $\bX$ or on the sparsity of $\bB^*$.

\begin{thm}\label{t:slow}
If $\lambda \geq \frac1{n}\|\bX^{\top}\bE\|_{\infty,2}$, then
$$
\frac1{n}\|\bX(\widehat \bB - \bB^*)\|_F^2 \leq \big(\lambda + \frac1{n}\|\bX^{\top}\bE\|_{\infty,2}\big) \|\bB^*\|_{1,2}.
$$
If, in addition,  $\lambda \geq \frac2{n}\|\bX^{\top}\bE\|_{\infty,2}$, then $
\|\widehat \bB\|_{1,2} \leq 3 \|\bB^*\|_{1,2}.
$
\end{thm}

The first part of Theorem~\ref{t:slow} is a well-known deterministic bound for lasso and group-lasso, see for example \citep{Lederer:2016wa,Rigollet:2011fz}. The second part of Theorem~\ref{t:slow} shows that the mixed $\ell_1/\ell_2$ norm of $\widehat \bB$ can be bounded by the $\ell_1/\ell_2$ norm of $\bB^*$ up to a constant. The latter, in particular, allows to bound expected out-of-sample prediction error following the proof similar to \citep{Chatterjee:2013vs} for the constrained lasso case. In the constrained formulation, the bound on the norm of $\widehat \bB$ is immediate by choosing a constraint parameter that is as large as $\|\bB^*\|_{1,2}$. We show that by choosing the tuning parameter $\lambda$ large enough, similar bound holds for penalized formulation. Combining the norm bound with the in-sample prediction error bound leads to the deterministic bound on expected out-of-sample prediction error. The bound on estimation error follows by assuming positive definiteness of population marginal covariance matrix $\bSigma_T$.

\begin{cor}\label{t:predbound} Let $\lambda \geq \frac{2}{n}\|\bX^{\top}\bE\|_{\infty,2}$.
 Then
$$
\Tr\{(\widehat \bB - \bB^*)^{\top}\bSigma_T(\widehat \bB - \bB^*)\} \leq \frac32\lambda\|\bB^*\|_{1,2}+ 16\|\bB^*\|_{1,2}^2\|\frac1{n}\bX^{\top}\bX-\bSigma_T\|_{\infty}.
$$
If, in addition, $\lambda_{\min}(\bSigma_T) > 0$, then
$$
\|\widehat \bB - \bB^*\|_F^2 \leq \frac3{2\lambda_{\min}(\bSigma_T)}\lambda\|\bB^*\|_{1,2}+ \frac{16}{\lambda_{\min}(\bSigma_T)}\|\bB^*\|_{1,2}^2\|\frac1{n}\bX^{\top}\bX-\bSigma_T\|_{\infty}.
$$
\end{cor}

The bounds of Corollary~\ref{t:predbound} are deterministic, and therefore depend on $\bX$ via $\|n^{-1}\bX^{\top}\bX - \bSigma_T\|_{\infty}$. In  Section~\ref{s:slow}, Theorem~\ref{t:slow_prob}, we provide the corresponding probabilistic bounds which rely on  concentration inequality for $n^{-1}\|\bX^{\top}\bE\|_{\infty,2}$ (Theorem~\ref{t:Xtepsilon}) and concentration inequality for $\|n^{-1}\bX^{\top}\bX - \bSigma_T\|_{\infty}$ (Lemma~\ref{l:SigmaInfBound}).

\subsection{Sparsity bounds}

To derive the sparsity bound, we make additional assumption on $\bB^*$.

\begin{assumption}[Sparsity]\label{a:sparsity} $\bB^*$ is row-sparse with the support $S = \{j:\|\be_j^{\top}\bB^*\|_2\neq 0\}$ with $s = \card(S)$.
\end{assumption}

\noindent
As in lasso, we also use restricted eigenvalue condition on the design matrix \citep{Bickel:2009eb}.

\begin{defin}[Restricted eigenvalue condition]\label{d:RE} A $q \times p$ matrix $\bQ$ satisfies restricted eigenvalue condition \textrm{RE}(s, c) with parameter $\gamma_{\bQ} = \gamma(s, c, \bQ)$ if for all sets $S \subset \{1,\dots, p\}$ with $\card(S) \leq s$, and for all $\ba \in \mathcal{C}(S,c)= \{\ba\in \R^{p}: \|\ba_{S^c}\|_1\leq c\|\ba_S\|_1\}$ it holds that
$$
\|\bQ\ba\|_2^2 \geq \frac{\|\ba_S\|_2^2}{\gamma_{\bQ}}.
$$
\end{defin}

In the group-lasso case, this condition is generalized to allow for the mixed $\ell_1/\ell_2$ norms, see for example \citep{Lounici:2011fl}. In penalized optimal scoring, the generalization is needed when the number of classes $K>2$. 

\begin{defin}[Group restricted eigenvalue condition]\label{d:REgroup} A $q \times p$ matrix $\bQ$ satisfies restricted eigenvalue condition \textrm{RE}(s, c, K) with parameter $\gamma_{\bQ} = \gamma(s, c, K, \bQ)$ if for all sets $S \subset \{1,\dots, p\}$ with $\card(S) \leq s$, and for all $\bA \in \mathcal{C}(S,c,K)= \{\bA\in \R^{p\times (K-1)}: \|\bA_{S^c}\|_{1,2}\leq c\|\bA_S\|_{1,2}\}$ it holds that
$$
\|\bQ\bA\|_F^2 \geq \frac{\|\bA_S\|_F^2}{\gamma_{\bQ}}.
$$
\end{defin}

When $K=2$, Definitions~\ref{d:RE} and~\ref{d:REgroup} coincide. We next state deterministic sparsity bounds that hold whenever $\bX$ satisfies restricted eigenvalue condition.

\begin{thm}\label{t:fast} Under Assumption~\ref{a:sparsity}, if $\lambda \geq \frac2{n}\|\bX^{\top}\bE\|_{\infty,2}$ and $n^{-1/2}\bX$ satisfies $\textrm{RE}(s,3,K)$ with parameter $\gamma_{\bX} = \gamma(s,3,K,n^{-1/2}\bX)$, then
$$
\frac1{n}\|\bX(\bB^* - \widehat \bB)\|_{F}^2 \leq \frac9{4}\gamma_{\bX} s\lambda^2;\quad  \|\widehat \bB - \bB^*\|_F \leq \frac{15}{2}\gamma_{\bX}\sqrt{s}\lambda\quad \mbox{and}\quad \|\widehat \bB - \bB^*\|_{1,2} \leq 6\gamma_{\bX} s\lambda.
$$
\end{thm}

The bounds of Theorem~\ref{t:fast} are well-known for lasso and group-lasso, see for example \citep{Bickel:2009eb,Negahban:2012fe,Hastie:2015wu}. Compared to these results, our interest is in expected out-of-sample prediction error, and we provide corresponding bound in Corollary~\ref{t:fast_p}. This bound can be obtained in two ways. On the one hand, we can use triangle inequality as in Corollary~\ref{t:predbound}. On the other hand, since restricted eigenvalue condition allows to directly bound estimation error, we can use that bound through the maximal eigenvalue of $\bSigma_T$. If the maximal eigenvalue of $\bSigma_T$ can be treated as constant, the second approach leads to tighter probabilistic bounds.

\begin{cor}\label{t:fast_p} Under Assumption~\ref{a:sparsity}, if $\lambda \geq \frac2{n}\|\bX^{\top}\bE\|_{\infty,2}$ and $n^{-1/2}\bX$ satisfies $\textrm{RE}(s,3, K)$ with parameter $\gamma_{\bX} = \gamma(s,3,K, n^{-1/2}\bX)$, then
$$
\Tr\{(\widehat \bB - \bB^*)^{\top}\bSigma_T(\widehat \bB - \bB^*)\} \leq \min\Big\{\frac9{4}\gamma_{\bX} s\lambda^2 + 36\gamma_{\bX}^2s^2\lambda^2\|\frac1{n}\bX^{\top}\bX-\bSigma_T\|_{\infty}, \lambda_{\max}(\bSigma_T)57\gamma_{\bX}^2s\lambda^2\Big\}.
$$
\end{cor}

\section{Probabilistic bounds}\label{s:prob}
Both Corollary~\ref{t:predbound} and~\ref{t:fast_p} rely on the deterministic condition for $\lambda$, that is $\lambda \geq \frac1{n}\|\bX^{\top}\bE\|_{\infty,2}$. Therefore, to derive corresponding probabilistic bounds, we need to derive a concentration bound for $\frac1{n}\|\bX^{\top}\bE\|_{\infty,2}$. This bound is provided in Section~\ref{s:error} and is the central result of the paper. The corresponding probabilistic slow-rate and fast-rate bounds are stated in Sections~\ref{s:slow} and~\ref{s:fast}.


\subsection{Concentration bound} \label{s:error}

There are several difficulties in deriving the concentration bound for $n^{-1}\|\bX^{\top}\bE\|_{\infty,2}$ in the context of penalized optimal scoring. First, both $\bY$ and $\bX$ are random, and the linear model for $\bY$ doesn't hold. Secondly, $\bX$ and $\bE$ are not independent. These challenges prevent application of lasso results, and therefre require new derivations. For this, we make the following assumptions.

\begin{assumption}[Class probabilities]\label{a:p}
$\pr(\bx_i \in \Ccal_k) = \pi_k$ for $k=1,\dots, K$ with $0<\pi_{\min}\leq \pi_{k} \leq \pi_{\max}<1$.
\end{assumption}


\begin{assumption}[Normality]\label{a:norm}
\label{a:data}
 $\bx_i|\bx_i \in \Ccal_k \sim \Ncal(\bmu_k, \bSigma_W)$ for all $k=1,\dots,K$ with $\bmu=\sum_{k=1}^K\pi_k\bmu_k = 0$.
\end{assumption}

\begin{assumption}[Sample size]\label{a:sample} $\log p = o(n)$ 
\end{assumption}

Assumption~\ref{a:p} requires prior group probabilities to be of the same order so that $n_k$ grows with $n$ for each $k$. 
Assumption~\ref{a:norm} is typical in linear discriminant analysis, however it can be relaxed to sub-gaussianity without affecting the rates. The normality allows to express the constants in the rates through the variance terms rather than sub-gaussian parameters, which we find easier to interpret. Without loss of generality, we assume that the overall mean $\bmu$ is zero. In practice, we always column-center data matrix $\bX$.  Assumption~\ref{a:sample} is a typical scaling for $n$ and $p$ in high-dimensional statistics.

Throughout, we use $\sigma^2_j$ to denote the diagonal elements of within-class covariance matrix $\bSigma_W$, and define
$$
\tau := \max_{j=1,\dots,p}\sqrt{\sigma_j^2 + \max_k \mu_{kj}^2}.
$$

\begin{thm}\label{t:Xtepsilon} Let $\lambda_0 = C\tau\sqrt{\frac{(K-1)\log (p\eta^{-1})}{n}}$ for some $\eta\in (0,1)$ and constant $C>0$. Under Assumptions~\ref{a:p}--\ref{a:sample}
$$
\pr\Big(\frac1n\|\bX^{\top}\bE\|_{\infty,2} \leq \lambda_0\Big) \geq 1 - \eta.
$$
\end{thm}

Theorem~\ref{t:Xtepsilon} provides a scaling of tuning parameter with respect to the number of classes $K$, the sample size $n$ and the number of variables $p$. While $\bX$ is random, and $\bX$ and $\bE$ are not independent, the scaling is the same up to constants as in lasso with fixed design, see for example \citep{Buhlmann:2011wj,Hastie:2015wu} and references therein. 

 We provide the sketch of the proof here to emphasize the new ideas. Since the linear model for $\bY$ does not hold, we explicitly take $\bE=\bY-\bX\bB^* = \bY - \bX\bSigma_T^{-1}\bDelta$ and use triangle inequality:
$$
\frac1{n}\|\bX^{\top}\bE\|_{\infty,2} = \|\frac1{n}\bX^{\top}\bY - \frac1{n}\bX^{\top}\bX\bSigma_T^{-1}\bDelta\|_{\infty,2} \leq \|\frac1{n}\bX^{\top}\bY - \bDelta\|_{\infty,2} + \|\bDelta - \frac1{n}\bX^{\top}\bX\bSigma_T^{-1}\bDelta\|_{\infty,2}.
$$
For the first term, we take advantage of the exact form of the optimal scores derived in Lemma~\ref{l:tildetheta} as well as tail inequality for quadratic forms of gaussian random vectors \citep{Hsu:2012cs}. For the second term, we prove that under Assumption~\ref{a:p}--\ref{a:norm}, the elements of $\bX$ are marginally sub-gaussian with parameter at most $\tau$ (Lemma~\ref{l:Xsubgaussian}) and derive element-wise concentration bound for the covariance matrix of random vector with sub-gaussian elements (Lemma~\ref{l:SigmaInfBound}). A particular feature of the bound in Theorem~\ref{t:Xtepsilon} is that the constant $C>0$ is not dependent on the model parameters $\bSigma_W$, $\bmu_k$ or $\pi_k$. The bound depends on the model parameters only through $\tau$, the dimension $p$ and the number of classes $K$. This is in contrast with the results of \citep{Li:2017kb} for $K=2$ case, where the constant implicitly depends on the model parameters and $\lambda_{\max}(\bSigma_W)$ in particular. To derive the explicit dependence through $\tau$,  we exploite the matrix decomposition of total covariance matrix $\bSigma_T$ from \citep{Gaynanova:2016wk}, the Woodbury matrix identity, and the bound on $\|\bSigma_T^{1/2}\|_{\infty}$ through $\|\bSigma_T\|_{\infty}$ based on concavity arguments. The full proof of Theorem~\ref{t:Xtepsilon} is in the Appendix.

\subsection{Slow rate bound}\label{s:slow}

In this section, we derive the slow rate bounds for out-of-sample prediction and $\ell_2$ estimation consistency of~\eqref{eq:scoringK} by combining the penalty bounds of Section~\ref{s:slow} with concentration bound of Theorem~\ref{t:Xtepsilon}.

\begin{thm}\label{t:slow_prob} If $\lambda \geq C\tau\sqrt{\frac{(K-1)\log (p)}{n}}$ for some constant $C>0$, then under Assumptions~\ref{a:p}--\ref{a:sample}
$$
\Tr\{(\widehat \bB - \bB^*)^{\top}\bSigma_T(\widehat \bB - \bB^*)\} = O_p\left[\Big\{1 \vee \tau\|\bB^*\|_{1,2}\Big\}\tau\|\bB^*\|_{1,2}\sqrt{\frac{(K-1)\log p}{n}}\right].
$$
If, in addition, $\lambda_{\min}(\bSigma_T) > 0$, then
$$
\|\widehat \bB - \bB^*\|_F^2 = O_p\left[\Big\{1 \vee \tau\|\bB^*\|_{1,2}\Big\}\frac{\tau\|\bB^*\|_{1,2}}{\lambda_{\min}(\bSigma_T)}\sqrt{\frac{(K-1)\log p}{n}}\right].
$$
\end{thm}

When $K=2$, these results mimic Theorem~1 in \citep{Chatterjee:2013vs} for the constrained lasso. Here we analyze the penalized formulation, and rely heavily on Theorem~\ref{t:Xtepsilon}, which required separate derivations for optimal scoring problem. The bound allows both the number of features $p$ and the number of classes $K$ to grow with $n$. If $\|\bB^*\|_{1,2}$ is a constant, the prediction consistency is achieved as long as $n \gg (K-1)\log p$. Otherwise, $\|\bB^{*}\|_{1,2}$ is allowed to grow at a rate no faster than $\{n/(K-1)\log p\}^{1/4}$. This scaling is suboptimal compared to what would be expected in lasso with fixed design, $\{n/(K-1)\log p\}^{1/2}$, and this discrepancy is a result of considering out-of-sample rather than in-sample prediction error. We refer to \citep{Chatterjee:2013vs} for further discussion.

\subsection{Fast rate bound}\label{s:fast}
In this section, we derive the fast rate bounds for out-of-sample prediction and $\ell_2$ estimation consistency of~\eqref{eq:scoringK} by combining the sparsity bounds of Section~\ref{s:fast} with concentration bound of Theorem~\ref{t:Xtepsilon} and restricted eigenvalue condition on the marginal covariance matrix $\bSigma_T$. The latter allows us to establish that restricted eigenvalue condition holds for random $n^{-1/2}\bX$ with high probability. For clarity, we assume that $\lambda_{\max}(\bSigma_T)$ is a constant so that the minimum in Corollary~\ref{t:fast} is achieved with the second bound.

We present the results for the case $K=2$ and $K>2$ separately. When $K=2$, 
we use \citep{Zhou:2009wba,Rudelson:2013jw} to show that $\textrm{RE}(s,c)$ holds with high probability for sub-gaussian matrices. 

\begin{thm}\label{t:fast_prob_p} Under Assumptions~\ref{a:sparsity}--\ref{a:sample}, if $K=2$, $\lambda = C\tau\sqrt{\frac{\log p}{n}}$ for some constant $C>0$, $s\log p=o(n)$, $\bSigma_T^{1/2}$ satisfies $\textrm{RE}(s, 9)$ and $\gamma = \gamma(s, 3, \bSigma_T^{1/2})$ according to Definition~\ref{d:RE}, then
\begin{align*}
(\widehat \bbeta - \bbeta^*)^{\top}\bSigma_T(\widehat \bbeta - \bbeta^*) &= O_p\left\{\lambda_{\max}(\bSigma_T)\tau^2\gamma^2\frac{s\log p}{n}\right\};\\
\|\widehat \bbeta - \bbeta ^*\|^2_2 &= O_p\Big(\tau^2\gamma^2\frac{s\log p}{n}\Big).
\end{align*}
\end{thm}


When $K>2$, we need to consider a more general condition $\textrm{RE}(s, c, K)$. We conjecture that the results of~\citep{Zhou:2009wba, Rudelson:2013jw} can be generalized to this condition, however the explicit proof is outside of the scope of this paper. For technical clarity, we instead bound $\gamma(s,c,K,n^{-1/2}\bX)$ through the element-wise maximum $\|n^{-1}\bX^{\top}\bX - \bSigma\|_{\infty}$ as in \citep{VanDeGeer:2009we}. This approach, however, leads to sub-optimal scaling of $s$ ($s^2\log p=o(n)$) compared to the $K=2$ case ($s\log p=o(n)$). This scaling is not present directly in the bounds, but rather is needed to ensure that $\gamma_{\bX}$ can be bounded by $\gamma$.

\begin{thm}\label{t:fast_prob_pK} Under Assumptions~\ref{a:sparsity}--\ref{a:sample}, if $\lambda = C\tau\sqrt{\frac{(K-1)\log p}{n}}$ for some constant $C>0$, $s^2\log p=o(n)$, $\bSigma_T^{1/2}$ satisfies $\textrm{RE}(s, 3, K)$ and $\gamma = \gamma(s, 3, K,\bSigma_T^{1/2})$ according to Definition~\ref{d:REgroup}, then
\begin{align*}
\Tr\{(\widehat \bB - \bB^*)^{\top}\bSigma_T(\widehat \bB - \bB^*)\} &= O_p\left\{\lambda_{\max}(\bSigma_T)\tau^2\gamma^2\frac{(K-1)s\log p}{n}\right\};\\
\|\widehat \bB - \bB ^*\|_F^2 &= O_p\Big(\tau^2\gamma^2\frac{(K-1)s\log p}{n}\Big).
\end{align*}
\end{thm}

Comparing Theorem~\ref{t:slow_prob} with Theorem~\ref{t:fast_prob_p} reveals that the key differences are in the use $\|\bB^*\|_{1,2}$ instead of cardinality $s$, and slower rate $\sqrt{\log p/n}$ compared to $\log p/n$ thus justifying commonly used slow-rate and fast-rate bounds terminology. The main advantage of Theorem~\ref{t:slow_prob} is the lack of sparsity assumption and restricted eigevalue condition. For more discussion on the advantages and disadvantages of these two bounds we refer to \citep{Dalalyan:2017je, Bien:2018jt}. Our main goal here is to show that penalized multi-class optimal scoring achieves the same consistency guarantees as lasso despite the lack of linear model for $\bY$ and dependency between $\bX$ and residuals $\bE$, and this is demonstrated via statements of Theorems~\ref{t:slow_prob} and~\ref{t:fast_prob_pK}. As with Theorem~\ref{t:slow_prob}, Theorem~\ref{t:fast_prob_pK} allows the number of classes $K$ to grow with $n$.

\section{Discussion}\label{sec:discus}

There has been significant progress in understanding the consistency of lasso and group-lasso estimators in linear regression \citep{Greenshtein:2004iz, Bunea:2007gr, Bickel:2009eb, Zhang:2009ck, Rigollet:2011fz,  Dalalyan:2017je,Lounici:2011fl,Yuan:2006wj}. These results can not be applied to penalized optimal scoring problem despite the similarity between corresponding optimization problems. The key difficulty is the lack of linear model for $\bY$, and the dependency between the random covariates in $\bX$ and the residual terms in $\bE$. In this work we overcome these challenges by using sub-exponential concentration bounds, and exploiting the decomposition of marginal covariance matrix $\bSigma_T$. While for clarity we focus on the linear optimal scoring and penalties of group-lasso type, the underlying technique can be used as a building block for investigating consistency of other problems, for example tensor discriminant analysis \citep{Pan:2018ff} or optimal scoring with weighted group-lasso penalty \citep{Merchante:2012vk}.
In our treatment of the fast rate bound for the multi-class case, we rely on group restricted eigenvalue condition for the random design matrix with sub-gaussian entries. When $K=2$, the existing results of \citep{Zhou:2009wba,Rudelson:2013jw} show the condition holds with high probability without affecting the rates. When $K>2$, the results of \citep{Zhou:2009wba,Rudelson:2013jw} do not strictly apply due to the different form of the cone constraint in Definition~\ref{d:REgroup}, but we conjecture that the same conclusions hold. It would be of interest to have a formal justification for this conjecture.

\newcommand{\Appendix}{\appendix\def\thesection{Appendix~\Alph{section}}\def\thesubsection{\Alph{section}.\arabic{subsection}}}
\section*{Appendix}
\begin{appendix}
\Appendix
\renewcommand{\theequation}{A.\arabic{equation}}
\renewcommand{\thesubsection}{A.\arabic{subsection}}
\setcounter{equation}{0}

\subsection{Technical proofs}

In this section, we prove the results stated in the main text. We use $C, C_1, C_2, C_3, \dots > 0$ to denote absolute constants that do not depend on model parameters. Their values may change from line to line.

\begin{proof}[Proof of Lemma~\ref{t:poptim}] Let $f = f(\bTheta, \bB)$ denote the objective function in~\eqref{eq:penopscore} and let $(\bTheta^*, \bB^*)$ be the global solution to~\eqref{eq:penopscore}, that is
$$
f^* = f(\bTheta^*, \bB^*) = \frac1{2n}\|\bZ\bTheta^* - \bX\bB^*\|^2_F + \lambda \Pen(\bB^*) \leq f(\bTheta, \bB)
$$
for all $\bB$ and all $\bTheta$ that satisfy the constraints. 
Since both $\widetilde \bTheta$ and $\bTheta^*$ satisfy the constraints, there exist orthogonal matrix $\bR\in\R^{(K-1)\times (K-1)}$ such that $\bTheta^*\bR =\widetilde \bTheta$. Let $\widetilde \bB = \bB^*\bR$, then using orthogonal invariance of the penalty function
\begin{align*}
f(\widetilde \bTheta,\widetilde \bB) &= \frac1{2n}\Tr\{(\bZ\bTheta^* - \bX\bB^*)\bR\bR^{\top}(\bZ\bTheta^* - \bX\bB^*)^{\top}\} + \lambda \Pen(\bB^*\bR) \\
&= \frac1{2n}\Tr\{(\bZ\bTheta^* - \bX\bB^*)(\bZ\bTheta^* - \bX\bB^*)^{\top}\} + \lambda \Pen(\bB^*) = f^*,
\end{align*}
that is the pair $(\widetilde \bTheta, \widetilde \bB)$ also attains global minimum. Then $f^* =f(\widetilde \bTheta, \widetilde \bB)\leq f(\widetilde \bTheta, \bB)$ for all $\bB$, that is
$$
\widetilde\bB = \bB_{\widetilde \Theta} = \argmin_{\bB}\left\{\frac1{2n}\|\bZ\widetilde \bTheta - \bX\bB\|^2_F + \lambda \Pen(\bB)\right\},
$$
and the pair $(\widetilde\bTheta, \bB_{\widetilde\bTheta})$ attains global minimum of~\eqref{eq:penopscore}.
\end{proof}

\begin{proof}[Proof of Lemma~\ref{l:tildetheta}]For any $l$, $j$ with $j>l$
\begin{align*}
\frac1{n}\bTheta_l^{\top}\bZ^{\top}\bZ \bTheta_l&=\left(\left\{\sqrt{\frac{n_{l+1}}{\sum_{i=1}^ln_i\sum_{i=1}^{l+1}n_i}}\right\}_{l},\quad-\sqrt{\frac{\sum_{i=1}^ln_i}{n_{l+1}\sum_{i=1}^{l+1}n_i}},\quad \{0\}_{K-1-l}\right)\bpm
n_1&0&\dots&0\\
0&n_2&\dots&0\\
&&\dots&\\
0&0&\dots&n_K
\epm\\
&\quad \times
\left(\left\{\sqrt{\frac{n_{l+1}}{\sum_{i=1}^ln_i\sum_{i=1}^{l+1}n_i}}\right\}_{l},\quad-\sqrt{\frac{\sum_{i=1}^ln_i}{n_{l+1}\sum_{i=1}^{l+1}n_i}},\quad\{ 0\}_{K-1-l}\right)^{\top}\\
&=\left(\sqrt{\frac{n_1^2n_{l+1}}{\sum_{i=1}^ln_i\sum_{i=1}^{l+1}n_i}},\dots,\sqrt{\frac{n_l^2n_{l+1}}{\sum_{i=1}^ln_i\sum_{i=1}^{l+1}n_i}},-\sqrt{\frac{n_{l+1}\sum_{i=1}^ln_i}{\sum_{i=1}^{l+1}n_i}}\right)\\
&\quad \times
\left(\left\{\sqrt{\frac{n_{l+1}}{\sum_{i=1}^ln_i\sum_{i=1}^{l+1}n_i}}\right\}_{l},\quad-\sqrt{\frac{\sum_{i=1}^ln_i}{n_{l+1}\sum_{i=1}^{l+1}n_i}}\right)^{\top}\\
&=\frac{n_{l+1}}{\sum_{i=1}^{l+1}n_i\sum_{i=1}^{l+1}n_i}\sum_{i=1}^l n_i+\frac{\sum_{i=1}^l n_i}{\sum_{i=1}^{l+1}n_i}=\frac{\sum_{i=1}^{l+1}n_i}{\sum_{i=1}^{l+1}n_i}=1 
\end{align*}
and
\begin{align*}
\frac1{n} \bTheta_l^{\top}\bZ^{\top}\bZ \bTheta_j&=\left(\left\{\sqrt{\frac{n_{l+1}}{\sum_{i=1}^ln_i\sum_{i=1}^{l+1}n_i}}\right\}_{l},\quad-\sqrt{\frac{\sum_{i=1}^ln_i}{n_{l+1}\sum_{i=1}^{l+1}n_i}},\quad\{0\}_{K-1-l}\right)\bpm
n_1&0&\dots&0\\
0&n_2&\dots&0\\
&&\dots&\\
0&0&\dots&n_K
\epm\\
&\quad \times
\left(\left\{\sqrt{\frac{n_{j+1}}{\sum_{i=1}^jn_i\sum_{i=1}^{j+1}n_i}}\right\}_{j},\quad-\sqrt{\frac{\sum_{i=1}^jn_i}{n_{j+1}\sum_{i=1}^{j+1}n_i}},\quad\{0\}_{K-1-j}\right)^{\top}\\
&=\left(\sqrt{\frac{n_1^2n_{l+1}}{\sum_{i=1}^ln_i\sum_{i=1}^{l+1}n_i}},\dots,\sqrt{\frac{n_l^2n_{l+1}}{\sum_{i=1}^ln_i\sum_{i=1}^{l+1}n_i}},\quad-\sqrt{\frac{n_{l+1}\sum_{i=1}^ln_i}{\sum_{i=1}^{l+1}n_i}},\quad \{0\}_{j-l}\right)\\
&\quad \times
\left(\left\{\sqrt{\frac{n_{j+1}}{\sum_{i=1}^jn_i\sum_{i=1}^{j+1}n_i}}\right\}_{j},\quad-\sqrt{\frac{\sum_{i=1}^jn_i}{n_{j+1}\sum_{i=1}^{j+1}n_i}}\right)^{\top}\\
&=\sum_{t=1}^l\sqrt{\frac{n_t^2n_{l+1}n_{j+1}}{\sum_{i=1}^ln_i\sum_{i=1}^{l+1}n_i\sum_{i=1}^jn_i\sum_{i=1}^{j+1}n_i}}-\sqrt{\frac{n_{l+1}\sum_{i=1}^ln_in_{j+1}}{\sum_{i=1}^{l+1}n_i\sum_{i=1}^jn_i\sum_{i=1}^{j+1}n_i}}\\
&=\sum_{t=1}^ln_t\sqrt{\frac{n_{l+1}n_{j+1}}{\sum_{i=1}^ln_i\sum_{i=1}^{l+1}n_i\sum_{i=1}^jn_i\sum_{i=1}^{j+1}n_i}}-\sum_{t=1}^ln_t\sqrt{\frac{n_{l+1}n_{j+1}}{\sum_{i=1}^ln_i\sum_{i=1}^{l+1}n_i\sum_{i=1}^jn_i\sum_{i=1}^{j+1}n_i}}\\
&=0.
\end{align*}
Finally, 
\begin{align*}
\frac1{\sqrt{n}}\bTheta_l^{\top}\bZ^{\top}\bZ\ones&=\left(\left\{\sqrt{\frac{n_{l+1}}{\sum_{i=1}^ln_i\sum_{i=1}^{l+1}n_i}}\right\}_{l},\quad-\sqrt{\frac{\sum_{i=1}^ln_i}{n_{l+1}\sum_{i=1}^{l+1}n_i}},\quad \{0\}_{K-1-l}\right)\bpm
n_1&0&\dots&0\\
0&n_2&\dots&0\\
&&\dots&\\
0&0&\dots&n_K
\epm\\
&\quad \times (\{1\}_l,1,\{1\}_{K-1-l})^{\top}\\
&=\left(\sqrt{\frac{n_1^2n_{l+1}}{\sum_{i=1}^ln_i\sum_{i=1}^{l+1}n_i}},\dots,\sqrt{\frac{n_l^2n_{l+1}}{\sum_{i=1}^ln_i\sum_{i=1}^{l+1}n_i}},\quad-\sqrt{\frac{n_{l+1}\sum_{i=1}^ln_i}{\sum_{i=1}^{l+1}n_i}}\right) \times (\{1\}_l, 1)^{\top}\\
&= \sum_{i=1}^l\sqrt{\frac{n_i^2n_{l+1}}{\sum_{i=1}^ln_i\sum_{i=1}^{l+1}n_i}} - \sqrt{\frac{n_{l+1}\sum_{i=1}^ln_i}{\sum_{i=1}^{l+1}n_i}}\\
& =  \sqrt{\frac{\sum_{i=1}^ln_in_{l+1}}{\sum_{i=1}^{l+1}n_i}} - \sqrt{\frac{n_{l+1}\sum_{i=1}^ln_i}{\sum_{i=1}^{l+1}n_i}} = 0.
\end{align*}
\end{proof}

\begin{proof}[Proof of Theorem~\ref{t:slow}]
The first part of the proof follows the proof of the ``slow-rate" bound for lasso, see for example \citep{Lederer:2016wa}. We reproduce the proof for completeness. Consider the KKT conditions for~\eqref{eq:opscore}:
$$
-\frac1{n}\bX^{\top}(\bY-\bX\widehat \bB) + \lambda \widehat \bPsi = 0,
$$
where $\widehat \bPsi$ is the subgradient of $\|\bB\|_{1,2}$ evaluated at $\widehat \bB$. It follows that
$$
\Tr[(\widehat \bB- \bB^*)^{\top}\{-\frac1{n}\bX^{\top}(\bY-\bX\widehat \bB) + \lambda \widehat \bPsi\}] = 0,
$$
and using $\bY = \bX\bB^* + \bY - \bX\bB^* = \bX\bB^* + \bE$
$$
\frac1{n}\|\bX(\bB^* - \widehat \bB)\|_F^2 - \langle \frac1{n}\bX^{\top}\bE, \widehat \bB - \widehat \bB^* \rangle + \lambda \langle \widehat \bPsi,\widehat \bB - \bB^* \rangle = 0.
$$
Since $\widehat \bPsi$ is the subgradient of the convex function $\|\bB\|_{1,2}$ evaluated at $\widehat \bB$, it follows that
$$
\|\widehat \bB\|_{1,2} \leq \|\bB^*\|_{1,2} + \langle \widehat \bPsi, \widehat \bB - \bB^*\rangle.
$$
Combining the above two displays leads to
\begin{align*}
\frac1{n}\|\bX(\bB^* - \widehat \bB)\|_F^2 - \langle \frac1{n}\bX^{\top}\bE, \widehat \bB - \widehat \bB^* \rangle + \lambda (\|\widehat \bB\|_{1,2} - \|\bB^*\|_{1,2})\leq 0.
\end{align*}
Rearranging the terms gives
\begin{align*}
   \frac1{n}\|\bX(\bB^* - \widehat \bB)\|_F^2 \leq  \langle \frac1{n}\bX^{\top}\bE, \widehat \bB - \widehat \bB^* \rangle + \lambda\|\bB^*\|_{1,2} - \lambda \|\widehat \bB\|_{1,2}.
\end{align*}
Using H\"older's inequality and triangle inequality gives
\begin{align*}
   \frac1{n}\|\bX(\bB^* - \widehat \bB)\|_F^2 &\leq  \frac1{n}\|\bX^{\top}\bE\|_{\infty,2} \|\widehat \bB - \widehat \bB^*\|_{1,2}  + \lambda\|\bB^*\|_1 - \lambda \|\widehat \bB\|_{1,2}\\
   &\leq (\frac1{n}\|\bX^{\top}\bE\|_{\infty,2} + \lambda)\|\bB^*\|_{1,2} + (\frac1{n}\|\bX^{\top}\bE\|_{\infty,2} - \lambda)\|\widehat \bB\|_{1,2}.
\end{align*}
Using $\lambda \geq \frac1{n}\|\bX^{\top}\bE\|_{\infty,2}$ completes the proof of the first part.

For the second part of the proof, using $\lambda \geq \frac2{n}\|\bX^{\top}\bE\|_{\infty,2}$ and the above display gives
\begin{equation*}
 \frac1{n}\|\bX(\bB^* - \widehat \bB)\|_F^2 \leq (\lambda/2 + \lambda)\|\bB^*\|_{1,2} + (\lambda/2 - \lambda) \|\widehat \bB\|_{1,2}.
\end{equation*}
Since the left-hand side is non-negative, rearranging the terms gives
$$
\frac{\lambda}{2}\|\widehat \bB\|_{1,2} \leq 3\frac{\lambda}{2}\|\bB^*\|_{1,2}.
$$
Since $\lambda > 0$, the result follows.
\end{proof}

\begin{proof}[Proof of Corollary~\ref{t:predbound}] By triangle and H\"older's inequalities
\begin{align*}
    \Tr\{(\widehat \bB - \bB^*)^{\top}\bSigma_T(\widehat \bB - \bB^*)\} \leq \frac1{n}\|\bX(\widehat \bB - \bB^*)\|_{F}^2 + \|\widehat \bB - \bB^*\|_{1,2}^2 \|\frac1{n}\bX^{\top}\bX-\bSigma_T\|_{\infty}.
\end{align*}
Applying Theorem~\ref{t:slow} leads to stated bound.
\end{proof}

\begin{proof}[Proof of Theorem~\ref{t:fast}]
From the proof of Theorem~\ref{t:slow}
$$
  \frac1{n}\|\bX(\bB^* - \widehat \bB)\|_F^2 \leq  \langle \frac1{n}\bX^{\top}\bE, \widehat \bB - \widehat \bB^* \rangle + \lambda\|\bB^*\|_{1,2} - \lambda \|\widehat \bB\|_{1,2}.
$$
Using $\lambda \geq \frac2{n}\|\bX^{\top}\bE\|_{\infty,2}$, H\"older's inequality and Assumption~\ref{a:sparsity},
\begin{align*}
    \frac1{n}\|\bX(\bB^* - \widehat \bB)\|_F^2 &\leq \frac{\lambda}{2}\|\widehat \bB - \bB^*\|_{1,2} + \lambda \|\bB^*\|_{1,2} - \lambda \|\widehat \bB\|_{1,2}\\
    &\leq \frac{\lambda}{2}\|\widehat \bB_S - \bB^*_S\|_{1,2} + \frac{\lambda}{2}\|\widehat \bB_{S^c}\|_{1,2}+ \lambda \|\bB^*\|_{1,2} - \lambda \|\widehat \bB_S\|_{1,2}- \lambda \|\widehat \bB_{S^c}\|_{1,2}\\
    &\leq \frac{\lambda}{2}\|\widehat \bB_S - \bB^*_S\|_{1,2} - \frac{\lambda}{2}\|\widehat \bB_{S^c}\|_{1,2}+ \lambda \|\bB^*\|_{1,2} - \lambda \|\bB^*\|_{1,2} + \lambda \|\widehat \bB^*_S- \widehat \bB_S\|_{1,2}\\
    &\leq \frac{3\lambda}{2}\|\widehat \bB_S - \bB^*_S\|_{1,2} - \frac{\lambda}{2}\|\widehat \bB_{S^c}\|_{1,2},
\end{align*}
where in the third step we use triangle inequality $\|\widehat \bB_S\|_{1,2} \geq \|\bB^*\|_{1,2} - \|\widehat \bB_S - \bB^*_S\|_{1,2}$. Since $\|\bX(\bB^* - \widehat \bB)\|_F^2 \geq 0$, it follows that
$
\lambda\|\widehat \bB_{S^c}\|_{1,2} \leq 3\lambda\|\widehat \bB_S - \bB^*_S\|_{1,2}. 
$
Since $\lambda>0$ and $\bB^*_{S^c}=0$ by Assumption~\ref{a:sparsity}, it follows that $\bA := \widehat \bB - \bB^*$ belongs to the cone $C(S,3,K)$ from Definition~\ref{d:REgroup}.

Since $\|\widehat \bB_S - \bB^*_S\|_{1,2} \leq \sqrt{s}\|\widehat \bB_S - \bB^*_S\|_F$, and $\frac1{n}\bX^{\top}\bX$ satisfies $\textrm{RE}(s, 3, K)$, from the above display
$$
\frac1{n}\|\bX(\bB^* - \widehat \bB)\|_F^2\leq  \frac{3\lambda}{2}\sqrt{s}\|\widehat \bB_S - \bB_S^*\|_F \leq \frac{3\lambda}{2}\sqrt{s}\sqrt{\gamma_{\bX}}\frac1{\sqrt{n}}\|\bX(\bB^* - \widehat \bB)\|_F.
$$
If $\|\bX(\widehat \bB - \bB^*)\|_F = 0$, the statement of the Theorem holds trivially. Otherwise dividing both sides by $\frac1{\sqrt{n}}\|\bX(\bB^* - \widehat \bB)\|_F$ gives
$$
\frac1{\sqrt{n}}\|\bX(\bB^* - \widehat \bB)\|_F \leq \frac{3}{2}\sqrt{\gamma_{\bX}}\sqrt{s}\lambda,
$$
which leads to
$$
\frac1{n}\|\bX(\bB^* - \widehat \bB)\|_F^2 \leq \frac9{4}\gamma_{\bX} s\lambda^2.
$$
Since $\frac1{n}\bX^{\top}\bX$ satisfies $\textrm{RE}(s, 3, K)$ and $\bA = \widehat \bB - \bB^*$ belongs to the cone $C(S,3, K)$,
\begin{align*}
    \|\widehat \bB - \bB^*\|_{1,2}=\|\bA\|_{1,2} = \|\bA_S\|_{1,2} + \|\bA_{S^c}\|_{1,2} \leq 4\|\bA_S\|_{1,2} \leq 4\sqrt{s} \|\bA_S\|_F\leq 4\sqrt{s}\sqrt{\gamma_{\bX}}\sqrt{\|\bX\bA\|_F^2/n} \leq 6s\lambda\gamma_{\bX}.
\end{align*}

Finally, to prove the bound on $\|\widehat \bB - \bB^*\|_F$, we follow derivations in Appendix~A.2 of~\citep{Zhou:2009wba}. Let $T_0$ correspond to the location of $s$ largest in euclidean norm rows of $\bA$, $T_1$ to the location of $s$ largest in euclidean norm rows of $\bA_{T_0^C}$, and so on for $T_2, T_3, \dots$. Then $\card(T_j) = s$, and
$$
\|\bA_{T_0^c}\|_{1,2} = \|\bA\|_{1,2} - \|\bA_{T_0}\|_{1,2} \leq \|\bA\|_{1,2} - \|\bA_{S}\|_{1,2} = \|\bA_{S^c}\|_{1,2}\leq 3\|\bA_S\|_{1,2} \leq 3\|\bA_{T_0}\|_{1,2} \leq 3\sqrt{s}\|\bA_{T_0}\|_F.
$$
Therefore
\begin{align*}
\|\bA\|_F &\leq \|\bA_{T_0}\|_F + \sum_{j\ge 1}\|\bA_{T_j}\|_F \leq \|\bA_{T_0}\|_F + \sum_{j\ge 1}\sqrt{s}\|\bA_{T_j}\|_{\infty,2} \\
&\leq \|\bA_{T_0}\|_F + \sum_{j\ge 0}\sqrt{s}\frac1{s}\|\bA_{T_j}\|_{1,2} \leq \|\bA_{T_0}\|_F + \frac1{\sqrt{s}}\|\bA\|_{1,2}.
\end{align*}
Using that $\frac1{n}\bX^{\top}\bX$ satisfies $\textrm{RE}(s, 3, K)$ and derived bounds leads to
$$
\|\widehat \bB - \bB^*\|_F = \|\bA\|_F \leq \sqrt{\gamma_{\bX}}\sqrt{\|\bX\bA\|_F^2/n} + \frac1{\sqrt{s}}6s\lambda \gamma_{\bX} \leq \frac{3}{2}\gamma_{\bX}\sqrt{s}\lambda + 6\gamma_{\bX}\sqrt{s}\lambda = \frac{15}{2}\gamma_{\bX}\sqrt{s}\lambda.
$$
\end{proof}

\begin{proof}[Proof of Corollary~\ref{t:fast_p}]
Using triangle and H\"older's inequalities
\begin{align*}
    \Tr\{(\widehat \bB - \bB^*)^{\top}&\bSigma_T(\widehat \bB - \bB^*)\} \\
    &=\Tr\{(\widehat \bB - \bB^*)^{\top}\frac1{n}\bX^{\top}\bX(\widehat \bB - \bB^*)\} + \Tr\{(\widehat \bB - \bB^*)^{\top}(\frac1{n}\bX^{\top}\bX - \bSigma_T)(\widehat \bB - \bB^*)\}\\
    &\leq \frac1{n}\|\bX(\widehat \bB - \bB^*)\|_F^2 + \|\widehat \bB - \bB^*\|_{1,2}^2 \|\frac1{n}\bX^{\top}\bX-\bSigma_T\|_{\infty}.
\end{align*}
Applying Theorem~\ref{t:fast} for $\lambda \geq \frac2{n}\|\bX^{\top}\bE\|_{\infty,2}$ gives
\begin{align*}
    \Tr\{(\widehat \bB - \bB^*)^{\top}\bSigma_T(\widehat \bB - \bB^*)\} \leq  \frac9{4}\gamma_{\bX} s\lambda^2 + 36\gamma_{\bX}^2s^2\lambda^2\|\frac1{n}\bX^{\top}\bX-\bSigma_T\|_{\infty}.
\end{align*}
On the other hand, using Theorem~\ref{t:fast} gives
$$
\Tr\{(\widehat \bB - \bB^*)^{\top}\bSigma_T(\widehat \bB - \bB^*)\} \leq \lambda_{\max}(\bSigma_T)\|\widehat \bB - \bB^*\|_F^2 \leq  \lambda_{\max}(\bSigma_T)57\gamma_{\bX}^2s\lambda^2.
$$
\end{proof}

\begin{proof}[Proof of Theorem~\ref{t:Xtepsilon}]
Consider
$$
\frac1{n}\|\bX^{\top}\bE\|_{\infty,2} = \|\frac1{n}\bX^{\top}\bY - \frac1{n}\bX^{\top}\bX\bB^*\|_{\infty,2} \leq \underbrace{\|\frac1{n}\bX^{\top}\bY - \bDelta\|_{\infty,2}}_{:=I_1} + \underbrace{\|\bDelta - \frac1{n}\bX^{\top}\bX\bSigma_T^{-1}\bDelta\|_{\infty,2}}_{:=I_2}.
$$
Consider $I_1$. From Lemma~\ref{l:D}, with probability at least $1-\eta$ for some constant $C>0$
$$
I_1 = \|\frac1{n}\bX^{\top}\bY - \bDelta\|_{\infty,2}\leq C\max_j \sigma_{j}\sqrt{\frac{(K-1)\log(p \eta^{-1})}{n}}.
$$
Consider $I_2$. Using $\|\bA\bB\bC\|_{\infty,2}\leq \|\bA\|_{\infty}\|\bB\|_{\infty}\|\bC\|_{\infty,2}$ \citep[Lemma~8]{Obozinski:2011ho} gives
\begin{equation}\label{eq:eterm}
\begin{split}
I_2 &= \|\bDelta - \frac1{n}\bX^{\top}\bX\bSigma_T^{-1}\bDelta\|_{\infty,2} = \|\bSigma_T^{1/2}(\bI-  \frac1{n}\bSigma_T^{-1/2}\bX^{\top}\bX\bSigma_T^{-1/2})\bSigma_T^{-1/2}\bDelta\|_{\infty,2}\\
&\leq \|\bSigma_T^{1/2}\|_{\infty}\|\bI-\frac1{n}\bSigma_T^{-1/2}\bX^{\top}\bX\bSigma_T^{-1/2}\|_{\infty}\|\bSigma_T^{-1/2}\bDelta\|_{\infty,2}.
\end{split}
\end{equation}
Consider $\|\bSigma_T^{-1/2}\bDelta\|_{\infty,2}$. Since $\bSigma_T = \bSigma_W + \bDelta\bDelta^{\top}$ \citep[Proposition~2]{Gaynanova:2016wk}, by Woodbury matrix identity
$
\bDelta^{\top}\bSigma_T^{-1}\bDelta = \bDelta^{\top}\bSigma_W^{-1}\bDelta(\bI + \bDelta^{\top}\bSigma_W^{-1}\bDelta)^{-1}.
$
Therefore,
\begin{align*}
\|\bSigma_T^{-1/2}\bDelta\|_{\infty,2} &= \max_j \|\be_j^{\top}\bSigma_T^{-1/2}\bDelta\|_2 = \max_j \sqrt{\be_j^{\top}\bSigma_T^{-1/2}\bDelta\bDelta^{\top}\bSigma_T^{-1/2}\be_j}\\
&\leq \|\bSigma_T^{-1/2}\bDelta\bDelta^{\top}\bSigma_T^{-1/2}\|_2  = \|\bDelta^{\top}\bSigma_T^{-1}\bDelta\|_2 = 
\|\bDelta^{\top}\bSigma_W^{-1}\bDelta(\bI + \bDelta^{\top}\bSigma_W^{-1}\bDelta)^{-1}\|_2\leq 1.
\end{align*}

Consider $\|\bSigma_T^{1/2}\|_{\infty}$. Let $\bSigma_T = \bU\bLambda \bU^{\top}$ be the eigendecomposition of $\bSigma_T$, then $\bSigma_T^{1/2} = \bU\bLambda^{1/2}\bU^{\top}$ is positive definite and $\|\bSigma_T^{1/2}\|_{\infty} = \max_j |\sum_{i=1}^p \sqrt{\lambda_i}u_{ji}^2|$. Since $f(x)=\sqrt{x}$ is concave and $\sum_{i=1}^pu_{ji}^2 = 1$ for all $j$, it follows that
$$
\|\bSigma_T^{1/2}\|_{\infty} = \max_j |\sum_{i=1}^p \sqrt{\lambda_i}u_{ji}^2| \leq \max_j \sqrt{\sum_{i=1}^p \lambda_iu_{ji}^2} \leq \sqrt{\|\bSigma_T\|_{\infty}} \leq \max_{j}\sqrt{\sigma_{jj}^2 + \max_{k}\mu_{kj}^2} = \tau,
$$
where the last inequality holds since $\bSigma_T = \bSigma_W + \sum_{k=1}^K\pi_k\bmu_k \bmu_k^{\top}$ for $\bmu=0$. 

Finally, from Lemma~\ref{l:Xsubgaussian}, all elements of $\bX\bSigma_T^{-1/2}$ are sub-gaussian with parameter $C$ that does not depend on $\bSigma_W$ or $\bmu_k$. Therefore, from Lemma~\ref{l:SigmaInfBound} with probability at least $1-\eta$
$$
\|\bI-\frac1{n}\bSigma_T^{-1/2}\bX^{\top}\bX\bSigma_T^{-1/2}\|_{\infty} \leq C_1\sqrt{\frac{\log(p\eta^{-1})}{n}}.
$$
Combining the above displays with~\eqref{eq:eterm} gives 
$$
I_2 \leq C_2 \tau\sqrt{\frac{\log(p\eta^{-1})}{n}}.
$$
Combining results for $I_1$ and $I_2$ gives
$$
\|\frac1{n}\bX^{\top}(\bY-\bX\bB^*)\|_{\infty,2} \leq C_3\tau\sqrt{\frac{(K-1)\log(p\eta^{-1})}{n}}
$$
with probability at least $1-\eta$ for some constant $C_3 >0$.
\end{proof}

\begin{proof}[Proof of Theorem~\ref{t:slow_prob}]
From Corollary~\ref{t:predbound}, if $\lambda \geq \frac2{n}\|\bX^{\top}\bE\|_{\infty,2}$,
$$
\Tr\{(\widehat \bB - \bB^*)^{\top}\bSigma_T(\widehat \bB - \bB^*)\} \leq \frac32\lambda\|\bB^*\|_{1,2}+ 16\|\bB^*\|_{1,2}^2\|\frac1{n}\bX^{\top}\bX-\bSigma_T\|_{\infty}.
$$
Applying Theorem~\ref{t:Xtepsilon} for $\frac{1}{n}\|\bX^{\top}\bE\|_{\infty,2}$ and Lemmas~\ref{l:Xsubgaussian} and~\ref{l:SigmaInfBound} for $\|\frac1{n}\bX^{\top}\bX-\bSigma_T\|_{\infty}$ leads to the desired statement.
\end{proof}

\begin{proof}[Proof of Theorem~\ref{t:fast_prob_p}]
By Lemma~\ref{l:ReCondition}, for $n \geq Cs\log(p/s)$, $n^{-1/2}\bX$ satisfies $\textrm{RE}(s,3)$ with probability at least $1-O(e^{-n})$ with $$\gamma_{\bX} = \gamma(s,3,n^{-1/2}\bX) \leq 2\gamma(s, 3, \bSigma_T^{1/2}) = 2\gamma. $$ 
The first bound follows by combining this with Corollary~\ref{t:fast_p} and Theorem~\ref{t:Xtepsilon}.
The second bound follows by combining this with the results of Theorem~\ref{t:fast} and Theorem~\ref{t:Xtepsilon}.
\end{proof}

\begin{proof}[Proof of Theorem~\ref{t:fast_prob_pK}]
By Lemmas~\ref{l:Xsubgaussian} and \ref{l:SigmaInfBound}, with probability at least $1-\eta$
$$
\|n^{-1}\bX^{\top}\bX - \bSigma_T\|_{\infty}\leq C\tau^2\sqrt{\frac{\log (p\eta^{-1})}{n}}.
$$
By Lemma~\ref{l:REgroup}, if $s\leq (32\gamma \|\bSigma_T - n^{-1}\bX^{\top}\bX\|_{\infty})^{-1}$, then $\gamma_{\bX} \leq 2 \gamma$. Therefore, using $s = o(\sqrt{n/\log p})$, Corollary~\ref{t:fast_p} and Theorem~\ref{t:Xtepsilon} gives that for $\lambda \geq C\tau\sqrt{\frac{(K-1)\log p}{n}}$
\begin{align*}
\Tr\{(\widehat \bB - \bB^*)^{\top}\bSigma_T(\widehat \bB - \bB^*)\} & = O_p\Big\{\lambda_{\max}(\bSigma_T)\tau^2\gamma^2\frac{(K-1)s\log p}{n}\Big\}.
\end{align*}
The second bound follows by combining the results of Theorem~\ref{t:fast} and Theorem~\ref{t:Xtepsilon}.
\end{proof}

\subsection{Additional lemmas}

\begin{lemma}\label{l:Xsubgaussian} Under Assumptions~\ref{a:p}--\ref{a:norm}, all elements of $\bX$ are sub-gaussian, that is
$$
\E e^{\lambda x_{ij}}\leq e^{\lambda^2\tau^2/2}\quad \mbox{for all}\quad \lambda \in \R;\quad i = 1,\dots, n;\quad j= 1,\dots, p;
$$
where $\tau = \max_{j=1,\dots,p}\sqrt{\sigma_j^2 + \max_k \mu_{kj}^2}$
with $\sigma_j^2$ being the diagonal elements of $\bSigma_W$. Similarly, all elements of $\bX\bSigma_T^{-1/2}$ are subgaussian with parameter $C>0$ that does not depend on $\bSigma_W$ or $\bmu_k$.
\end{lemma}

\begin{proof}
Since $\bx_i|\bx_i\in \Ccal_k \sim N(\bmu_k, \bSigma_W)$,
$$
\bx_{i} = \sum_{k=1}^K \bmu_{k}\Ind\{\bx_i \in \Gcal_k\} + \bSigma_W^{1/2}\bzeta_i = \bt_{1i} + \bt_{2i},
$$
where $\bzeta_i\sim \Ncal(0,\bI)$ and $\bt_{1i}$, $\bt_{2i}$ are independent random vectors.

Since $\bmu = 0$, $t_{1ij}$ is mean zero random variable with
$|t_{1ij}|\leq \max_k |\mu_{kj}|$, hence $t_{1ij}$ is sub-gaussian with parameter at most $\max_k |\mu_{kj}|$.
On the other hand, $\bt_{2i}$ is mean zero gaussian random vector with $\cov(\bt_{2i}) = \bSigma_{W}$. Hence, $\var(t_{2ij}) = \sigma^2_{j}$ for all $j$, and $t_{2ij}$ is sub-gaussian with parameter $\sigma_{j}$. Since $t_{1ij}$ and $t_{2ij}$ are independent,
\begin{align*}
\E(e^{\lambda x_{ij}}) = \E\{e^{\lambda(t_{1ij} + t_{2ij})}\} = \E(e^{\lambda t_{1ij}})\E(e^{\lambda t_{2ij}})\leq e^{\lambda^2\{\sigma_j^2 + \max_k \mu_{kj}^2\}/2}.
\end{align*}
Therefore, $x_{ij}$ is sub-gaussian with parameter $\tau_j = \sqrt{\sigma^2_{j} +\max_k \mu_{kj}^2}$. Letting $\tau=\max_j \tau_j$, all elements of $\bX$ are sub-gaussian with parameter at most $\tau$.

Similarly, 
$$
\bSigma_T^{-1/2}\bx_i = \bSigma_T^{-1/2}\sum_{k=1}^K \bmu_{k}\Ind\{\bx_i \in \Gcal_k\} + \bSigma_T^{-1/2}\bSigma_{W}^{1/2}\bzeta_i = \bu_{1i} + \bu_{2i}.
$$
Let $\bM=[\bmu_1\dots \bmu_k]\in \R^{p \times k}$, then
$$
\|\bu_{1i}\|_{\infty}= \|\bSigma_T^{-1/2}\sum_{k=1}^K \bmu_{k}\Ind\{\bx_i \in \Gcal_k\}\|_{\infty}\leq \|\bSigma_T^{-1/2}\bM\|_{\infty,2}\leq \|\bM^{\top}\bSigma_T^{-1}\bM\|_2.
$$
Let $\bPi = \diag(\pi_1,\dots, \pi_K)$, then $\bSigma_T = \bSigma_W + \bM\bPi \bM^{\top}$, and by Woodbury matrix identity 
$$
\|\bM\bSigma_T^{-1}\bM\|_{2} = \|\bM^{\top}\bSigma_W^{-1}\bM\bPi^{-1}(\bPi^{-1}+\bM^{\top}\bSigma_W^{-1}\bM)^{-1}\|_2 \leq C,
$$
where the last inequality uses Assumption~\ref{a:p}. It follows that all elements of $\bu_{1i}$ are bounded by at most $C$, hence are sub-gaussian with parameter at most $C$.
On the other hand, using $\bSigma_T = \bSigma_W + \bDelta\bDelta^{\top}$ \citep[Proposition~2]{Gaynanova:2016wk}, 
$$
\cov(\bu_{2i}) = \bSigma_T^{-1/2}\bSigma_W\bSigma_T^{-1/2}=\bSigma_T^{-1/2}(\bSigma_T - \bDelta\bDelta^{\top})\bSigma_T^{-1/2} = \bI - \bSigma_T^{-1/2}\bDelta\bDelta^{\top}\bSigma_T^{-1/2},
$$
therefore by Assumption~\ref{a:norm} all elements of $\bu_{2i}$ are sub-gaussian with parameter at most one. Since $\bSigma_T^{-1/2}\bx_i = \bu_{1i} + \bu_{2i}$, it follows that all elements of $\bSigma_T^{-1/2}\bx_i$ are subgaussian with parameter at most $C_1$ independent of $\bSigma_W$ and $\bmu_k$.
\end{proof}

\begin{lemma}\label{l:D}
Let $\bD=\frac1{n}\bX^{\top}\bY = \frac1{n}\bX^{\top}\bZ\widetilde\bTheta$, where $\widetilde \bTheta$ is from Lemma~\ref{l:tildetheta}, and let $\bDelta$ be as in~\eqref{eq:delta}. Undear Assumptions~\ref{a:p}--\ref{a:norm}, with probability at least $1-\eta$
$$
\|\bD - \bDelta\|_{\infty,2} \leq C\max_j\sigma_j\sqrt{\frac{(K-1)\log (p\eta^{-1})}{n}}.
$$
\end{lemma}
\begin{proof}
Using the definition of $\widetilde\bTheta$ and $\bZ$, it follows that the $l$th column of $\bD$ has the form
$$
\bD_l = \frac1{\sqrt{n}}\frac{\sqrt{n_{l+1}}\sum_{i=1}^l n_i(\bar{\bx}_i - \bar {\bx}_{l+1})}{\sqrt{\sum_{i=1}^l n_i\sum_{i=1}^{l+1}n_{i}}}.
$$
Using Assumptions~\ref{a:p}--\ref{a:norm} and Lemma~8 in \citep{Gaynanova:2015km}, $\bD$ has matrix-normal distribution with $\E(\bD) = \bDelta + o(1)$ and $\cov(\bD) = n^{-1}\bSigma_W + o(1)$. Applying the tail inequality for quadratic form of the gaussian random vector \citep[Proposition~1.1]{Hsu:2012cs} gives for all $t>0$
$$
\pr(\|\be_j^{\top}\bD - \be_j^{\top}\bDelta\|_2^2/\sigma_j^2 > (K-1) + 2\sqrt{(K-1)t} + 2t) \leq e^{-t}.
$$
Applying union bound over all $j\in\{1,\dots, p\}$ and taking large $t$ gives that with probability at least $1-\eta$
$$
\|\bD - \bDelta\|_{\infty,2} \leq C \max_j \sigma_j\sqrt{\frac{(K-1)\log(p \eta^{-1})}{n}}.
$$

\end{proof}

\begin{lemma}\label{l:ReCondition} Under Assumptions~\ref{a:sparsity}--\ref{a:sample}, if $\bSigma_T^{1/2}$ satisfies $\textrm{RE}(s,9)$, then for $n\geq Cs\log(p/s)$ with probability at least $1-O(e^{-n})$
$n^{-1/2}\bX$ satisfies $\textrm{RE}(s, 3)$ with parameter
$$
0<\gamma\Big(s,3, n^{-1/2}\bX\Big)\leq 2\gamma(s, 3, \bSigma_T^{1/2}).
$$
\end{lemma}
\begin{proof}
 Under Assumptions~\ref{a:p}--\ref{a:norm}, all elements of $\bX$ are sub-gaussian with marginal covariance matrix $\bSigma_T$. The result follows using the assumption on $\bSigma_T^{1/2}$ and applying Theorem~6 in \citep{Rudelson:2013jw} with $\delta = 1/2$.
\end{proof}

\begin{lemma}\label{l:REgroup} Let $\bSigma_T^{1/2}$ satisfy $\textrm{RE}(s, 3, K)$ with $\gamma = \gamma(s,3,K,\bSigma_T^{1/2})$. If $s \leq (32\gamma\|\bSigma_T - n^{-1}\bX^{\top}\bX\|_{\infty})^{-1}$, then $n^{-1/2}\bX$ satisfies $RE(s,3, K)$ with 
$$
0 < \gamma(s,3,K, n^{-1/2}\bX)\leq 2\gamma(s,3,K,\bSigma_T^{1/2}).
$$
\end{lemma}
\begin{proof}[Proof of Lemma~\ref{l:REgroup}] Since $\bSigma_T^{1/2}$ satisfies $\textrm{RE}(s, 3, K)$, for all $\bA\in \mathcal{C}(S,3,K)$
\begin{align*}
    \frac1{n}\Tr(\bA^{\top}\bX^{\top}\bX\bA) = \Tr(\bA^{\top}\bSigma_T\bA) + \Tr\{\bA^{\top}(\bSigma_T-n^{-1}\bX^{\top}\bX)\bA\} \geq \frac1{\gamma}\|\bA_S\|^2_F - \|\bA\|_{1,2}^2\|\bSigma_T-n^{-1}\bX^{\top}\bX\|_{\infty}.
\end{align*}
Since $\bA\in \mathcal{C}(S,3, K)$, 
$
\|\bA\|_{1,2} = \|\bA_S\|_{1,2} + \|\bA_{S^c}\|_{1,2}\leq 4\|\bA_S\|_{1,2}.
$
Therefore
\begin{align*}
 \frac1{n}\Tr(\bA^{\top}\bX^{\top}\bX\bA) &\geq \frac1{\gamma}\|\bA_S\|^2_F - 16\|\bA_S\|_{1,2}^2\|\bSigma_T-n^{-1}\bX^{\top}\bX\|_{\infty} \\
 &\geq \frac1{\gamma}\|\bA_S\|^2_F - 16 s\|\bA_S\|_F^2 \|\bSigma_T-n^{-1}\bX^{\top}\bX\|_{\infty} \\
 &\geq \frac1{\gamma}\|\bA_S\|^2_F - \frac1{2\gamma}\|\bA_S\|^2_F = \frac1{2\gamma}\|\bA_S\|_F^2,
\end{align*}
where we used the condition on $s$ in the last inequality. 
\end{proof}

\begin{lemma}\label{l:SigmaInfBound} Let $\bx_1,...,\bx_n\in\R^p$ be independent zero-mean random vectors with $\max_{j=1,\dots,p}\|x_{ij}\|_{\psi_2}\leq \tau$, $\cov(\bx_i) = \bSigma$, and let $\bX = [\bx_1 \dots \bx_n]^{\top}$. Under Assumption~\ref{a:sample}, for some constant $C>0$ and a fixed $\eta \in (0,1)$
$$
\|n^{-1}\bX^{\top}\bX - \bSigma\|_{\infty}\leq C\tau^2\sqrt{\frac{\log (p\eta^{-1})}{n}}
$$
with probability at least $1-\eta$.
\end{lemma}
\begin{proof} The statement is equivalent to Lemma~F.2 in \citep{Neykov:2015to}. For completeness, we provide the full proof.
Let $t_{ikj} = x_{ik}x_{ij}$. Since $\max_{j=1,\dots,p}\|x_{ij}\|_{\psi_2}\leq \tau$, applying Cauchy-Shwartz inequality together with \citep[Lemma~5.14]{Vershynin:2010vk} leads to
$$
\|t_{ikj}\|_{\psi_1} \leq \sqrt{\|x_{ik}^2\|_{\psi_1}\|x_{ij}^2\|_{\psi_1}} \leq \sqrt{2\|x_{ik}\|^2_{\psi_2}2\|x_{ij}\|^2_{\psi_2}} = 2\|x_{ik}\|_{\psi_2}\|x_{ij}\|_{\psi_2}\leq 2 \tau^2.
$$
That is, $t_{ikj}$ is sub-exponential with parameter $2\tau^2$. Moreover, $\|t_{ikj} - \sigma_{kj}\|_{\psi_2} = \|t_{ikj}-\E(t_{ikj})\|_{\psi_2}\leq 2\|t_{kij}\|_{\psi_2} \leq 4 \tau^2$ is also sub-exponential with parameter $4\tau^2$. Applying Bernstein's bound \citep[Corollary~5.17]{Vershynin:2010vk} together with the union bound leads to
$$
\pr(\|n^{-1}\bX^{\top}\bX - \bSigma\|_{\infty} \geq \varepsilon) \leq 2p^2\exp[-C\min(\varepsilon^2/16\tau^4, \varepsilon/4\tau^2)n].
$$
for some constant $C>0$. Letting $\varepsilon = C_1\tau^2\sqrt{\frac{\log (p\eta^{-1})}{n}}$ for fixed $\eta \in (0,1)$ and using Assumption~\ref{a:sample} gives
$\|n^{-1} \bX^{\top}\bX - \bSigma\|_{\infty}\leq \varepsilon$ with probability at least $1 - \eta$.
\end{proof}

\end{appendix}

\section*{Acknowledgement}

Gaynanova's research was supported by National Science Foundation grant DMS-1712943.

\bibliographystyle{imsart-number}
\bibliography{IrinaReferences}
\end{document}